\documentclass{amsart}
\usepackage{amsfonts}
\usepackage{amsthm}
\usepackage{amssymb}
\usepackage{amsmath}

\newtheorem{theorem}{\sc Theorem}
\newtheorem{lemma}[theorem]{\sc Lemma}
\newtheorem{proposition}[theorem]{\sc Proposition}
\newtheorem{corollary}[theorem]{\sc Corollary}

\begin{document}

\title{Limit Groups and Automorphisms of $\kappa$-Existentially Closed Groups}
\author{Burak KAYA, Mahmut KUZUCUO\u{G}LU, Patrizia LONGOBARDI*, Mercede MAJ }

\thanks{{\scriptsize
\hskip -0.4 true cm MSC(2020): Primary: 20B27; Secondary: 20B35, 20E36, 20F28.
Received: \\
$*$Corresponding author \\
 This work was supported by the National Group for Algebraic and Geometric Structures, and their Applications (GNSAGA-INdAM), Italy. The second author is grateful to the Department of Mathematics of the University of Salerno for its hospitality and support, while this investigation was carried out. }}

\maketitle

\begin{center}

Dedicated to Prof. Otto H. Kegel for his 90th birthday
\end{center}

\begin{center}

\today
\end{center}

\begin{abstract}
The structure of automorphism groups of $\kappa$-existentially closed groups  are studied
 by Kaya-Kuzucuo\u{g}lu in 2022.
 It was proved that Aut(G) is the union of subgroups of level preserving automorphisms  and $|Aut(G)|=2^\kappa$
 whenever $\kappa$ is an inaccessible cardinal and $G$ is the unique
 $\kappa$-existentially closed group of
  cardinality $\kappa$. The cardinality of the automorphism group of a $\kappa$-existentially
  closed group of cardinality $\lambda>\kappa$ is asked in Kourovka Notebook Question 20.40.
   Here we answer positively the promised case $\kappa=\lambda$ namely:
   If $G$ is a $\kappa$-existentially closed group of cardinality $\kappa$,
   then $|Aut(G)|=2^{\kappa}$. We also answer Kegel's question
    on universal groups, namely: For any uncountable cardinal $\kappa$,
  there exist universal groups of cardinality $\kappa$.
\end{abstract}

\section{Introduction}

Let $\kappa$ be an infinite cardinal. A group $G$ is called a $\kappa$-existentially closed group if
 every system of less than $\kappa$ many equations and in-equations with coefficients from $G$
 which has a solution in an overgroup $H\geq G$ already has a solution in $G$. The structure
  of the automorphism group of a $\kappa$-existentially closed group and the cardinality of the
 automorphism group of the unique $\kappa$-existentially closed group of cardinality $\kappa$
 for an inaccessible cardinal $\kappa$ is studied in \cite{kaya-k}.
The following open question was asked in \cite{kourovka} Kourovka-Notebook  Question 20.40:
 Let $G$ be a $\kappa$-existentially closed group of cardinality $\lambda > \kappa$.
  Is it true that $|Aut(G)|=2^{\lambda}$ ? While the general question in Kourovka-Notebook \cite{kourovka} remains open, the positive answer in the equality case $\lambda=\kappa$, which
 was promised in Kourovka-Notebook, is answered by using the results in \cite[Chapter I]{Kueker75}. More specifically, we prove the following.

 \begin{theorem}\label{mainauttheorem} Let $\kappa$ be a regular cardinal and $G$ be a
$\kappa$-existentially closed group of cardinality $\kappa$. Then $|Aut(G)|=2^{\kappa}$.
\end{theorem}

A group $U$ is called a \textit{universal group}
in the class of all groups if every group $H$ with $|H|\leq |U|$ is isomorphic
 to a subgroup of $U$. In \cite[Theorem C]{kegel-2008}, Kegel asked the following question:
 For which cardinals $\kappa$ does there exist universal groups of cardinality $\kappa$? Here, under GCH, we also answer this question using $\kappa$-existentially closed groups.

\begin{proposition}[GCH]\label{universal}  For any uncountable cardinal $\kappa$,
  there exist universal groups of cardinality $\kappa$.
\end{proposition}

\section{Strongly $\kappa$-partially isomorphism and automorphism groups}

It turns out that some algebraic facts regarding $\kappa$-existentially closed groups can be
 obtained as a corollary, to the fact that such groups are strongly $\kappa$-partially isomorphic.
 Let us quickly remind this model-theoretic notion which first appeared in \cite{Kueker75}.

Given two structures $\mathfrak{A}$ and  $\mathfrak{B}$ over the same language with universes
$A$ and $B$ respectively, we say that $\mathfrak{A}$ and
 $\mathfrak{B}$ are \textit{strongly $\kappa$-partially isomorphic},
  written $\mathfrak{A} \cong_{\kappa}^s \mathfrak{B}$, if there exists a non-empty set
  $I$ of isomorphisms from substructures of $\mathfrak{A}$ to substructures of $\mathfrak{B}$ such that

\begin{enumerate}
\item[(i)]  For every $f \in I$ and every $X \subseteq A$ with $|X|<\kappa$,
there exists $g \in I$ such that $X \subseteq dom(g)$ and $f \subseteq g$,
\item[(ii)]  For every $f \in I$ and every $Y \subseteq B$
with $|Y|<\kappa$, there exists $g \in I$ such that $Y \subseteq ran(g)$ and $f \subseteq g$,
\item[(iii)]  $I$ is closed under unions of chains of length less than the cofinality of $\kappa$ namely $cf(\kappa)$.
\end{enumerate}

While the notion of being $\kappa$-partially isomorphic,
 a weakening of this notion that only requires the first two items (i) and (ii),
 is well-studied in the context of infinitary logic, e.g. see \cite{Dickmann75},
 this stronger notions seems to not have gotten enough attention.

Strongly $\kappa$-partial isomorphism allows one to carry out back-and-forth arguments
in a general setting. Indeed, we have that if $\mathfrak{A} \cong_{\kappa}^s \mathfrak{B}$ and $|A|=|B|=\kappa$,
 then $\mathfrak{A}$ and $\mathfrak{B}$ are isomorphic \cite[Chapter I, Theorem 2.1]{Kueker75}.
For the case of $\kappa$-existentially closed groups, we have the following.

\begin{theorem}\label{stronglyisomorphic} Let $\kappa$ be an uncountable cardinal.
Then any two $\kappa$-existentially closed groups are strongly $\kappa$-partially isomorphic.
\end{theorem}
\begin{proof} Let $G$ and $H$ be two uncountable  $\kappa$-existentially closed groups.
Choose $I$ to be the set of all isomorphisms from subgroups of $G$ of cardinality less than
$\kappa$ to subgroups of $H$.

Let $f \in I$ and $X \subseteq G$ be such that $|X|<\kappa$. Set $K=\langle dom(f),X \rangle$.
 Then we have $|K|<\kappa$ and hence $K$ embeds into $H$ by
 \cite[Lemma 2.1]{kk}, say, $\varphi: K \rightarrow H$ is an embedding.
 Since $M=\varphi[dom(f)]$ and $N=ran(f)$ are isomorphic subgroups of $H$ via the isomorphism
 $\psi: M \rightarrow N$ given by $\psi(x)=(f \circ \varphi^{-1})(x)$, it follows from \cite[Lemma 2.4]{kk}
  that there exists $h \in H$ such that $\varphi(x)^ h=f(x)$ for all $x \in dom(f)$.
  Choose the map $g=\iota_h \circ \varphi$ where $\iota_h$ is the inner automorphism induced by the element $h\in H$.
   Then clearly $f \subseteq g$ and $X \subseteq dom(g)$.
  This shows that the first requirement (i) in the definition of $\cong_{\kappa}^s$ is satisfied.
  The proof of the second requirement (ii) is similar.

For the last requirement (iii), let $\mu<cf(\kappa)$ and let $\{f_{\alpha}\}_{\alpha<\mu} \subseteq I$
be a chain of isomorphisms. Then we have that $f=\bigcup_{\alpha<\mu} f_{\alpha}$
is an isomorphism and moreover,
$$|dom(f)|=\left|\bigcup_{\alpha<\mu} dom(f_{\alpha})\right|<\kappa$$
since $|dom(f_{\alpha})|<\kappa$
for all $\alpha<\mu<cf(\kappa)$. Therefore $f \in I$, which completes the proof that $G$
and $H$ are strongly $\kappa$-partially isomorphic.
\end{proof}

For uncountable $\kappa$, using Theorem \ref{stronglyisomorphic}
together with the  observation in the previous paragraph before Theorem \ref{stronglyisomorphic},
one can show that any two $\kappa$-existentially closed groups of  cardinality $\kappa$ are
 isomorphic, which was first proven in \cite[Theorem 2.7]{kk}. The argument there,
 which had been independently discovered by the authors, was essentially
  the proof of Theorem \ref{stronglyisomorphic} carried out together with the proof of
   \cite[Chapter I, Theorem 2.1]{Kueker75}.

So as a corollary, we obtain the following.

\begin{corollary} Let $\kappa$ be an uncountable cardinal. Then any two  $\kappa$-existentially
 closed groups of  cardinality $\kappa$ are isomorphic.
\end{corollary}

Another important consequence of Theorem \ref{stronglyisomorphic}
is that any $\kappa$-existentially closed group of cardinality $\kappa$ has $2^{\kappa}$
automorphisms whenever $\kappa$ is regular. This important result is
 obtained by the application of the following theorem due to Kueker.

 \begin{theorem}\cite[Chapter I, Theorem 2.3]{Kueker75}\label{automorphismfact} Let $\mathfrak{A}$ be of
cardinality $\kappa$. Suppose that $\mathfrak{A}$ is strongly $\kappa$-partially isomorphic
to some $\mathfrak{B}$ with $|B|>\kappa$. Then $\mathfrak{A}$ has at least $\kappa^{cf(\kappa)}$ automorphisms.
\end{theorem}

There are arbitrarily large $\kappa$-existentially closed groups which can be obtained
by limits of regular representations see, \cite[Corollary 8]{KaKeKu}. Consequently, Theorem \ref{stronglyisomorphic} and Theorem \ref{automorphismfact} together imply Theorem \ref{mainauttheorem}.

Kueker seems to have intentionally omitted the proof of Theorem \ref{automorphismfact}
leaving it to a future paper. Unfortunately, we were unable to locate the relevant paper of the author.
 Therefore, for completeness, we will give a direct proof of
 Theorem \ref{mainauttheorem} using the ideas in the proof of \cite[Chapter I, Theorem 1.4]{Kueker75}
 and employing basic facts about $\kappa$-existentially closed groups to handle certain points,
  which would have required more effort if we were to prove Theorem \ref{automorphismfact}
   in its full generality.

Before we proceed, let us explain a notation that will be used in the proof.
 For two groups $G,H$ and two sequences $\{g_{\alpha}\}_{\alpha<\delta}$ and $\{h_{\alpha}\}_{\alpha<\delta}$
  of elements of $G$ and $H$ respectively, we will write $(G,\{g_{\alpha}\}_{\alpha<\delta})\cong(H,\{h_{\alpha}\}_{\alpha<\delta})$
  to mean that there is an isomorphism
   $\varphi: G \rightarrow H$ such that $\varphi(g_{\alpha})=h_{\alpha}$ for all $\alpha<\delta$.

Recall that by \cite[Corollary 7]{KaKeKu}, if there exists a
$\kappa$-existentially closed group of cardinality $\kappa$, then $\kappa$ is a regular cardinal.

\begin{lemma}\label{nontrivialextend} Let $\kappa$ be an uncountable cardinal.
 Let $G$ be a $\kappa$-existentially
closed group and $\{g_{\alpha}\}_{\alpha<\delta}$ be a sequence of elements of $G$ where $\delta<\kappa$.
Then for any $\{g'_{\alpha}\}_{\alpha<\delta}$ with
\[ (G,\{g_{\alpha}\}_{\alpha<\delta}) \cong (G,\{g'_{\alpha}\}_{\alpha<\delta})\]
there exist $k,h,h' \in G$ with $h \neq h'$ such that
\[(G,\{g_{\alpha}\}_{\alpha<\delta},k) \cong (G,\{g'_{\alpha}\}_{\alpha<\delta},h)\]
\[(G,\{g_{\alpha}\}_{\alpha<\delta},k) \cong (G,\{g'_{\alpha}\}_{\alpha<\delta},h')\]
\end{lemma}

\begin{proof}
 Let $(g'_{\alpha})_{\alpha<\delta}$ be such that
$(G,\{g_{\alpha}\}_{\alpha<\delta}) \cong (G,\{g'_{\alpha}\}_{\alpha<\delta})$, say, $\varphi: G \rightarrow G$
be an automorphism such that
$\varphi(g_{\alpha})=g'_{\alpha}$ for all $\alpha<\delta$.
 By \cite[Lemma 3.6]{kk} we have $|C_G(\{g_{\alpha}\}_{\alpha<\delta})|=|G|$ so,
there exist
$1 \neq g \in C_G(\{g_{\alpha}\}_{\alpha<\delta})$ and also by \cite[Lemma 2.10]{kk} we have
$|G|=|G:C_G(g)|$, and so
any nontrivial coset representative $k$ of $C_G(g)$ in $G$ satisfies $k^g\neq k$.
So  we can choose some element $k \in G$ such that $k^g \neq k$.
Set $h=\varphi(k)$ and $h'=\varphi(k^g)$.
Then clearly $h \neq h'$ and moreover, we have
\[(G,\{g_{\alpha}\}_{\alpha<\delta},k) \cong (G,\{g'_{\alpha}\}_{\alpha<\delta},h)\]
\[(G,\{g_{\alpha}\}_{\alpha<\delta},k) \cong (G,\{g'_{\alpha}\}_{\alpha<\delta},h')\]
by the automorphisms $\varphi$ and $\varphi \circ \iota_g$ respectively.
\end{proof}

\section{Proof of Theorem \ref{mainauttheorem}}

We are now ready to prove Theorem \ref{mainauttheorem} by transfinitely
generalizing the construction in the proof of \cite[Chapter I, Theorem 1.4]{Kueker75}
using $\kappa$-existential closedness.

\begin{proof}[Proof of Theorem \ref{mainauttheorem}] Let $G$ be a $\kappa$-existentially
closed group of cardinality $\kappa$, let, $\{a_{\alpha}\}_{\alpha<\kappa}$ be an  enumeration
of the elements of $G$. For each $f \in {^{\kappa}2}$ where $^{\kappa}2$
denotes the set of functions from $\kappa$ to $\{0,1\}$, we will obtain an
 automorphism $\varphi_f: G \rightarrow G$. Fix $f \in {^{\kappa}2}$.
 We will construct partial maps $\varphi_{f,\alpha}$ for every $\alpha<\kappa$ so that
\begin{itemize}
\item[(a).] If the domain of $\varphi_{f,\alpha}$ is $\{c_{\beta}\}_{\beta<\delta}$ where $\delta<\kappa$, then
\[ (G,\{c_{\beta}\}_{\beta<\delta}) \cong (G,\{\varphi_{f,\alpha}(c_{\beta})\}_{\beta<\delta})\]
\item[(b).] $\varphi_{f,\alpha}$ extends $\varphi_{f,\beta}$ whenever $\beta<\alpha$.
\item[(c).] The domain and range of $\varphi_{f,\alpha}$ both include $\{a_{\beta}\}_{\beta<\alpha}$.
\item[(d).] The domain and range of $\varphi_{f,\alpha}$ are of cardinality less than $\kappa$.
\end{itemize}
The maps $\varphi_{f,\alpha}$ will be constructed by induction on $\alpha<\kappa$ as follows.
\begin{itemize}
\item Set $\varphi_{f,0}=\emptyset$.
\item Let $\alpha<\kappa$. Suppose as inductive hypothesis that $\varphi_{f,\alpha}$
 has been constructed so that (a), (b), (c) and (d) hold for $\varphi_{f,\alpha}$.
  We will construct $\varphi_{f,\alpha+1}$ so that (a), (b), (c) and (d) hold for $\varphi_{f,\alpha+1}$.

By assumption, we know that
$(G,\{c_{\beta}\}_{\beta<\delta}) \cong (G,\{\varphi_{f,\alpha}(c_{\beta})\}_{\beta<\delta})$,
 say, via the automorphism  $\psi: G \rightarrow G$.
  Clearly we have that
\[(G,\{c_{\beta}\}_{\beta<\delta},a_{\alpha},\psi^{-1}(a_{\alpha})) \cong (G,\{\varphi_{f,\alpha}(c_{\beta})\}_{\beta<\delta},\psi(a_{\alpha}),a_{\alpha}).\]
Applying Lemma \ref{nontrivialextend}, we can obtain $k,h,h' \in G$ with $h \neq h'$
such that
\[(G,\{c_{\beta}\}_{\beta<\delta},a_{\alpha},\psi^{-1}(a_{\alpha}),k) \cong (G,\{\varphi_{f,\alpha}(c_{\beta})\}_{\beta<\delta},\psi(a_{\alpha}),a_{\alpha},h)\]
\[(G,\{c_{\beta}\}_{\beta<\delta},a_{\alpha},\psi^{-1}(a_{\alpha}),k) \cong (G,\{\varphi_{f,\alpha}(c_{\beta})\}_{\beta<\delta},\psi(a_{\alpha}),a_{\alpha},h').\]
The partial map $\varphi_{f,\alpha+1}$ is now defined as follows.

\item The domain of $\varphi_{f,\alpha+1}$ is $\{c_{\beta}\}_{\beta<\delta} \cup \{a_{\alpha},\psi^{-1}(a_{\alpha}),k\}$ where the latter elements are added to the end of the sequence.
\item $\varphi_{f,\alpha+1}(c_{\beta})$ extends $\varphi_{f,\alpha}$ on $\{c_{\beta}\}_{\beta<\delta}$
\item $\varphi_{f,\alpha+1}(a_{\alpha})=\psi(a_{\alpha})$
\item $\varphi_{f,\alpha+1}(\psi^{-1}(a_{\alpha}))=a_{\alpha}$
\item $\varphi_{f,\alpha+1}(k)=\begin{cases} h & \text{ if } f(\alpha)=0\\ h' & \text{ if } f(\alpha)=1 \end{cases}$
\end{itemize}

It is now easily checked that (a), (b), (c) and (d) hold for $\varphi_{f,\alpha+1}$.

\item Let $\gamma<\kappa$ be a limit ordinal. Suppose as inductive hypothesis $\varphi_{f,\alpha}$
has been constructed for all
 $\alpha<\gamma$ so that (a), (b), (c) and (d) hold for $\varphi_{f,\alpha}$ for each $\alpha<\gamma$.
  We will construct $\varphi_{f,\gamma}$ so that (a), (b), (c) and (d) hold for $\varphi_{f,\gamma}$.

Set $\varphi_{f,\gamma}=\bigcup_{\alpha<\gamma} \varphi_{f,\alpha}$. Observe that,
since $\varphi_{f,\alpha}$'s are compatible with each other on their common domains
 by the inductive assumption, we have that $\varphi_{f,\gamma}$
is indeed a partial map with domain $K=\bigcup_{\alpha<\gamma} dom(\varphi_{f,\alpha})$.
 Let us enumerate $K$, say, $K=\{c_{\beta}\}_{\beta<\delta}$ for some cardinal  $\delta$.

That the requirements (b) and (c) hold for $\varphi_{f,\gamma}$ is immediate.
 To see that the requirement (d) holds for $\varphi_{f,\gamma}$,
 note that $\gamma<\kappa$ and each $dom(\varphi_{f,\alpha})$
 has cardinality less than $\kappa$ by the inductive assumption.
  It now follows from the regularity of $\kappa$ that $|K|<\kappa$.

By \cite[Lemma 2.4]{kk}, $|K|<\kappa$ implies that there exists an inner automorphism induced by
$k \in G$ such that $\varphi_{f,\gamma}(x)=x^k$ for all $x \in K$. Consequently, we have
\[ (G,\{c_{\beta}\}_{\beta<\delta}) \cong (G,\{\varphi_{f,\gamma}(c_{\beta})\}_{\beta<\delta})\]
\noindent via the automorphism
$\iota_k$. Therefore, the requirement (a) holds for $\varphi_{f,\gamma}$.
 This finishes the inductive construction.

We now define $\varphi_f = \bigcup_{\alpha<\kappa} \varphi_{f,\alpha}$.
It is straightforward to check that $\varphi_f$ is an automorphism of $G$
for if it were not, then there would be finitely many witnesses for its failure
 of being a homomorphism or a bijection. But these witnesses would have to appear
 in an early stage $\alpha<\kappa$ of the inductive construction.
However, this cannot happen by the requirements (a), (b) and (c).

Before we finish the proof, we will need one last observation.
 While the notation $\varphi_{f,\alpha}$ suggests that the construction of the map
  $\varphi_{f,\alpha}$ depends on $f$ itself, it actually depends on the restriction of
 $f$ to $\alpha$. Therefore, by choosing the same witnesses $\psi, k,h, h'$ for
 functions with the same restriction onto $\alpha$
 throughout the inductive construction,
 we may assume without loss of generality that if
 $f \upharpoonright \alpha = f' \upharpoonright \alpha$,
  then $\varphi_{f,\alpha}=\varphi_{f',\alpha}$.

Finally, let $f,f' \in {^{\kappa}2}$ be such that $f \neq f'$.
Pick the least $\alpha<\kappa$ such that $f(\alpha) \neq f'(\alpha)$.
 Then, $f \upharpoonright \alpha = f' \upharpoonright \alpha$ and hence
  $\varphi_{f,\alpha}=\varphi_{f',\alpha}$ by the previous observation.
This implies $\varphi_{f,\alpha+1} \neq \varphi_{f',\alpha+1}$
because we have chosen the same $k,h,h' \in G$ for both $f$ and $f'$ at stage
 $\alpha$ of the construction. This subsequently implies $\varphi_{f} \neq \varphi_{f'}$.
  Therefore, the map $f \mapsto \varphi_f$ is an injection from
 ${^{\kappa}2}$ to $Aut(G)$ showing that $|Aut(G)|=2^{\kappa}$.\end{proof}

 Observe that \cite[Corollary 5]{kaya-k} is an immediate corollary of Theorem \ref{mainauttheorem} since inaccessible cardinals are regular.\\

 \section{Levels of a $\kappa$-existentially closed group}

 Having established that any $\kappa$-existentially closed group $G$ of cardinality $\kappa$
has $2^{\kappa}$ automorphisms, we would like to point out that, our construction of these
automorphisms depend on a well-ordering of $G$. For an explicit construction of
 $2^{\kappa}$ automorphisms of a limit of regular representations of length $\kappa$,
 we refer the reader to \cite{kaya-k},
 where the first two authors established a way to produce
  $2^{\kappa}$ level preserving automorphisms.

Since stratifications of $\kappa$-existentially closed groups into ``levels" seem to naturally
 arise in the context, we would like to note the following related result that shows that one
  can stratify a $\kappa$-existentially closed group of cardinality $\kappa$ into levels that can be made as much existentially closed as possible.

\begin{theorem}[GCH]\label{main} Let $\kappa$ be an uncountable cardinal and
 $I=\{\lambda \leq \kappa: \lambda \text{ is regular}\}$.
  Let $G$ be a $\kappa$-existentially closed group of cardinality $\kappa$.
   Then $G$ has a chain $\{G_\alpha\}_{\alpha \in I}$ of subgroups
\[\langle 1 \rangle <G_{\aleph_0} <G_{\aleph_1}< \ldots <G_\lambda < \ldots\]
with $G=\bigcup_{\lambda \in I} G_{\lambda}$, where the subgroup $G_{\lambda}$
 is a $\lambda$-existentially closed group of cardinality $\lambda$
  for every regular cardinal $\lambda \leq \kappa$ with $G_{\kappa}=G$.
\end{theorem}

In order to prove Theorem \ref{main}, we shall need the following lemma,
 a primitive version of which essentially appeared in \cite[Theorem 3]{Scott51}
  without cardinality constraints.

\begin{lemma}\label{scottslemmaextended} Let $\lambda \leq \kappa$ be
uncountable cardinals such that $\lambda^{<\lambda}=\lambda$. Let $H$ be
 a $\kappa$-existentially closed group and let $K \leqslant H$
 be a subgroup of cardinality $\lambda$. Then there exists a
  $\lambda$-existentially closed group  $\overline{K}$
   of cardinality $\lambda$ such that $K \leqslant \overline{K} \leqslant H$.
\end{lemma}

\begin{proof} Since $\lambda^{<\lambda}=\lambda$, we have that $\lambda$
is regular and moreover, there are only $\lambda$ many systems of less than
 $\lambda$ many equations and inequations with coefficients
  in a group of cardinality $\lambda$. We shall define a chain
    $\{\overline{K}_{\alpha}\}_{\alpha < \lambda}$ of subgroups
 of $H$ of cardinality $\lambda$ by transfinite recursion as follows.

Set $\overline{K}_0=K$. Let $\alpha<\lambda$ and suppose that the subgroups
 $\{\overline{K}_{\beta}\}_{\beta<\alpha}$ have been defined
  so that they form a chain of subgroups of cardinality $\lambda$.
   Let $\{s^{\alpha}_\delta\}_{\delta<\lambda}$ be an enumeration
   of systems of less than $\lambda$ equations and
    inequations with coefficients in $\bigcup_{\beta<\alpha} \overline{K}_{\beta}$.
     We shall define a subgroup $\overline{K}_{\alpha}$
  in which we can solve all these systems whenever it is possible.

In order to define $\overline{K}_{\alpha}$, we shall carry out another transfinite recursion to define a chain of subgroups $\{J^{\alpha}_{\delta}\}_{\delta < \lambda}$ of $H$ as follows:

\begin{itemize}
\item Set $J^{\alpha}_{0}=\bigcup_{\beta<\alpha} \overline{K}_{\beta}$.
\item Let $\delta<\lambda$ and suppose that $J^{\alpha}_{\delta}$ has been defined. If the system $s^{\alpha}_{\delta}$ has a solution in $H$, say $S \subseteq H$, then  $|S|<\lambda$, we define $J^{\alpha}_{\delta+1}=\langle J^{\alpha}_{\delta}, S \rangle$. Otherwise, we define $J^{\alpha}_{\delta+1}=J^{\alpha}_{\delta}$

\item Let $\gamma < \lambda$ be a limit ordinal. We define $J^{\alpha}_{\gamma}=\bigcup_{\delta<\gamma} J^{\alpha}_{\delta}$.
\end{itemize}

We now define $\overline{K}_{\alpha}=\bigcup_{\gamma<\alpha} J^{\alpha}_{\gamma}$.
 A straightforward transfinite induction together with the initial inductive hypothesis that $\{\overline{K}_{\beta}\}_{\beta<\alpha}$ are of cardinality $\lambda$
 shows that each $J^{\alpha}_{\gamma}$ is of cardinality
 $\lambda$ and hence $\overline{K}_{\alpha}$ is of cardinality $\lambda$,
 which completes the recursive construction.

Set $\overline{K}=\bigcup_{\alpha<\lambda} \overline{K}_{\alpha}$.
Since each $\overline{K}_{\alpha}$ is of cardinality $\lambda$,
 so is $\overline{K}$. In order to prove $\lambda$-existential closedness,
  let $s$ be a system of less than $\lambda$ many equations and inequations
  with coefficients in $\overline{K}$ that has a solution in an overgroup of
$\widehat{K} \geqslant \overline{K}$. Since $\lambda$ is regular, there
  must be some $\alpha<\lambda$ such that $\overline{K}_{\alpha}$
 contains all the constants appearing in $s$. But then $s=s^{\alpha+1}_{\delta}$
 for some $\delta<\lambda$. Observe that the system $s$ has a solution in the product
 $H \times \widehat{K}$ and hence has a solution in $H$ because $H$ is
 $\lambda$-existentially closed. It now follows from the successor step
 of the recursive construction that $s$ has a solution in
 $J^{\alpha+1}_{\delta+1} \leqslant \overline{K}_{\alpha} \leqslant \overline{K}$.\end{proof}

We are now ready to prove Theorem \ref{main}

\begin{proof}[Proof of Theorem \ref{main}] Before we start, let us recall that $\lambda^{<\lambda}=\lambda$
for all regular cardinals since we are assuming GCH. Thus Lemma \ref{scottslemmaextended}
 applies for all $\lambda \in I$.

Let $\{g_\alpha\}_{\alpha < \kappa}$ be an enumeration of $G$.
Let $K \mapsto \overline{K}$
be the operator in Lemma \ref{scottslemmaextended}
where the larger group is taken to be $G$.
Set $K_{\lambda}=\langle g_{\alpha}: \alpha < \lambda \rangle$
for each regular cardinal $\lambda \leq \kappa$.
By transfinite recursion on the well-ordered set $I$, define
\begin{itemize}
\item $G_0=\overline{K_{\aleph_0}}$ and
\item $G_{\lambda}=\overline{\langle K_{\lambda}, \bigcup_{\mu < \lambda} G_{\mu}} \rangle$
for every regular cardinal $\lambda \leq \kappa$.
\end{itemize}
Let us now show that the sequence $\{G_{\lambda}\}_{\lambda \in I}$
 is as required using transfinite induction on the well-ordered set $I$.
 As the base case of the induction, since $|K_{\aleph_0}|=\aleph_0$,
 we have that the group $G_0$ is an $\aleph_0$-existentially closed
  group of cardinality $\aleph_0$ by Lemma \ref{scottslemmaextended}.

Now let $\lambda \leq \kappa$ be a regular cardinal and suppose that $G_{\theta}$
is a $\theta$-existentially closed group of cardinality $\theta$
for all regular cardinals $\theta<\lambda$. By the inductive assumption,
 one can show that $\left|\langle K_{\lambda}, \bigcup_{\mu < \lambda} G_{\mu} \rangle\right|=\lambda$
  and hence we must have that
  $G_{\lambda}$ is a $\lambda$-existentially closed
  group of cardinality $\lambda$  by Lemma \ref{scottslemmaextended}.
   This completes the induction.
\end{proof}

The above proof indicates the following:
\begin{corollary}[GCH] Let $G$ be a $\kappa$-existentially closed group of cardinality
 $\lambda\geq \kappa$. Then
 $G$ has a local
 system consisting of $\mu$-existentially closed (and so, simple) groups for each
regular  cardinal $\mu \leq\kappa$.
\end{corollary}

\begin{proof} Let $G$ be a $\kappa$-existentially closed group of cardinality
 $\lambda\geq \kappa$. Then, by Lemma \ref{scottslemmaextended},
 one  may show that for each  regular cardinal $\mu \leq\kappa$,
  the existential closure of a subgroup  of cardinality $\mu$ exists and is of cardinality $\mu$.
 So every subgroup of cardinality $\mu$ is contained  in a
 $\mu$-existentially closed group of cardinality $\mu$.
 Hence $G$ has a local system consisting of $\mu$-existentially
 closed subgroups for any
regular cardinal $\mu$. As each $\mu$-existentially closed
group is simple, it follows that  $G$ has  a local system consisting of simple
 groups of cardinality  $\mu$.
\end{proof}

\section{Universal Groups}

Following Kegel \cite{kegel-2008}, a group $U$ is called a \textit{universal group}
in the class of all groups if every group $H$ with $|H|\leq |U|$ is isomorphic
 to a subgroup of $U$. By Lagrange's theorem,  every nontrivial finite universal group has order $2$. So we only talk about infinite universal groups.

 Kegel constructed examples of universal groups of cardinality
 $\kappa$ for limit cardinals $\kappa$ in \cite[Theorem C]{kegel-2008} and asked the following question:
 For which cardinals $\kappa$ does there exist universal groups of cardinality $\kappa$? Here we answer this question.

 By \cite[Theorem 14]{Neumann73}, there are $2^{\aleph_0}$ pairwise non-isomorphic two generated groups.
  As any countable group can contain at most countably many two generated subgroups, we may conclude that,
   an infinite universal group must have uncountable  cardinality. For uncountable cardinalities, we have Proposition \ref{universal}.

\begin{proof}[Proof of Proposition \ref{universal}]

It follows from Kegel \cite[Theorem C]{kegel-2008} that, starting with any countable group $G_0$ and iterating the right regular representations $G_{\alpha} \hookrightarrow G_{\alpha+1}=Sym(G_{\alpha})$ of symmetric groups $\gamma$ times for any limit ordinal $\gamma$, the regular limit group $G_{\gamma}$ is a universal group of cardinality $\aleph_{\gamma}$. Consequently, there exist universal groups of cardinality $\kappa$ for any singular cardinal $\kappa$ by Kegel's construction.

By \cite[Lemma 2.6]{kk}, an isomorphic copy of every group of
cardinality less than or equal to $\kappa$ can be embedded
in a $\kappa$-existentially closed group of cardinality $\kappa$.
 Hence every $\kappa$-existentially closed group of cardinality $\kappa$ is a universal group
of cardinality $\kappa$. The existence of $\kappa$-existentially closed groups of cardinality $\kappa$ for
each regular cardinality $\kappa$ is shown in \cite[Corollary 8]{KaKeKu}.
 (In particular, for each successor cardinal $\kappa$, there exist
  $\kappa$-existentially closed groups of cardinality $\kappa$.) It follows that, for any uncountable cardinality
$\kappa$, there exist universal groups of cardinality $\kappa$.

\end{proof}

While many of the universal groups obtained in Proposition \ref{universal} are existentially closed, one can obtain universal groups from other groups with different constructions so that some fundamental properties of existentially closed groups such as uniqueness and simpleness are not preserved.

Observe that if $U$ is a universal group, then every group $H \geqslant U$ with $|H|=|U|$ is again universal.
 So the direct product $\bigoplus_{\lambda} U$  of a universal group $U$ of cardinality $\kappa$ is also a universal group where  $\lambda \leq \kappa$. By Fuchs \cite[Theorem 89.2]{fuch2}, for every infinite cardinal $\mu$ less than the
 first strongly inaccessible cardinal,
 there exist $2^\mu$ torsion-free, abelian, indecomposable,
 pairwise non-isomorphic groups  $A_\alpha$, each of which has cardinality $\mu$.
 Then the groups $G_\alpha=U\times A_\alpha $, where $U$
 is a universal group of cardinality $\mu$, is a universal group of cardinality $\mu$. Thus we have $2^\mu$
 pairwise non-isomorphic universal groups of cardinality $\mu$.

 Moreover, one may show using \cite[Corollary 3.3]{bek} that,
 if $G$ is a $\kappa$-existentially closed group of cardinality $\kappa$,
 then the centralizer subgroup $C_G(H)$
 is a universal group of cardinality $\kappa$ for any subgroup $H$ of $G$ generated by fewer than $\kappa$ elements.



\bigskip
\textbf{Burak KAYA}

Department of Mathematics,

Middle East Technical University,

06800, Ankara, Turkey.

burakk@metu.edu.tr
\bigskip

\textbf{Mahmut KUZUCUO\u{G}LU}

Department of Mathematics,

Middle East Technical University,

06800, Ankara, Turkey.

 matmah@metu.edu.tr
\bigskip

\textbf{Patrizia LONGOBARDI}

Department of Mathematics

 University of Salerno

 Via Giovanni Paolo II, 132,

 84084 - Fisciano, Salerno ITALY

  plongobardi@unisa.it

 \bigskip

\textbf{Mercede MAJ}

Department of Mathematics

 University of Salerno

 Via Giovanni Paolo II,  132,

 84084 - Fisciano, Salerno ITALY

mmaj@unisa.it

\end{document}